\documentclass[sn-mathphys,Numbered]{sn-jnl}

\usepackage{graphicx}%
\usepackage{multirow}%
\usepackage{amsmath,amssymb,amsfonts}%
\usepackage{amsthm}%
\usepackage{mathrsfs}%
\usepackage[title]{appendix}%
\usepackage{xcolor}%
\usepackage{textcomp}%
\usepackage{manyfoot}%
\usepackage{booktabs}%
\usepackage{algorithm}%
\usepackage{algorithmicx}%
\usepackage{algpseudocode}%
\usepackage{listings}%

\theoremstyle{thmstyleone}%
\newtheorem{theorem}{Theorem}
\newtheorem{corollary}{Corollary}
\newtheorem{lemma}{Lemma}
\newtheorem{proposition}{Proposition}%

\theoremstyle{thmstyletwo}%
\newtheorem{example}{Example}%
\newtheorem{remark}{Remark}%
\newtheorem{property}{Property}
\theoremstyle{thmstylethree}%
\newtheorem{definition}{Definition}%

\begin{document}
	
	\title[Frames and Operators on quaternionic Hilbert spaces]{Frames and Operators on Quaternionic Hilbert Spaces}
	
	
	\author{\fnm{Najib} \sur{Khachiaa}}\email{khachiaa.najib@uit.ac.ma}
	\affil{\orgdiv{Laboratory Partial Differential Equations, Spectral Algebra and Geometry, Department of Mathematics}, \orgname{Faculty of Sciences, University Ibn Tofail}, \orgaddress{\city{Kenitra}, \country{Morocco}}}
	
	\abstract{The aim of this work is to study frame theory in quaternionic Hilbert spaces. We provide a characterization of frames in these spaces through the associated operators. Additionally, we examine frames of the form $\{Lu_i\}_{i \in I}$, where $L$ is a right $\mathbb{H}$-linear bounded operator and $\{u_i\}_{i \in I}$ is a frame.}

	
	
	\maketitle

\section{Introduction and preliminaries}
Frames in quaternionic Hilbert spaces provide a robust framework for analyzing and reconstructing signals within a higher-dimensional space. As an extension of the classical frame theory, these structures facilitate efficient data representation and processing. The unique properties of quaternionic spaces offer new avenues for exploration, particularly in applications such as signal processing and communications. This work aims to investigate the theoretical foundations of frames in quaternionic Hilbert spaces and their practical implications. The quaternionic field is an extension of the real and complex number systems, consisting of numbers known as quaternions. Quaternions are used to represent three-dimensional rotations and orientations, making them invaluable in computer graphics and robotics. Unlike real and complex numbers, quaternion multiplication is non-commutative, which adds complexity to their algebraic structure. Quaternions also provide a more efficient way to perform calculations in three-dimensional space, enhancing applications in physics and engineering.

\begin{definition}[The field of quaternions]
The non-commutative field of quaternions $\mathbb{H}$ is a four-dimensional real algebra with unity. In $\mathbb{H}$, $0$ denotes the null element  and $1$ denotes the identity with respect to multiplication. It also includes three so-called imaginary units, denoted by $i,j,k$. i.e.,
$$\mathbb{H}=\{a_0+a_1i+a_2j+a_3k:\; a_0,a_1,a_2,a_3\in \mathbb{R}\},$$
where $i^2=j^2=k^2=-1$, $ij=-ji=k$, $jk=-kj=i$ and $ki=-ik=j$. For each quaternion $q=a_0+a_1i+a_2j+a_3k$, we deifine the conjugate of $q$ denoted by $\overline{q}=a_0-a_1i-a_2j-a_2k \in \mathbb{H}$ and the module of $q$ denoted by $\vert q\vert $ as 
$$\vert q\vert =(\overline{q}q)^{\frac{1}{2}}=(q\overline{q})^{\frac{1}{2}}=\displaystyle{\sqrt{a_0^2+a_1^2+a_2^2+a_3^2}}.$$
For every $q\in \mathbb{H}$, $q^{-1}=\displaystyle{\frac{\overline{q}}{\vert q\vert ^2}}.$
\end{definition}

\begin{definition}[Right quaternionic vector space]
A right quaternioniq vector space $V$ is a linear vector space under right scalar multiplication over the field of quaternions $\mathbb{H}$, i.e., the right scalar multiplication 
$$\begin{array}{rcl}
V\times \mathbb{H} &\rightarrow& V\\
(v,q)&\mapsto& v.q,
\end{array}$$
satisfies the following for all $u,v\in V$ and $q,p\in \mathbb{H}$:
\begin{enumerate}
\item $(v+u).q=v.q+u.q$,
\item $v.(p+q)=v.p+v.q$,
\item $v.(pq)=(v.p).q$.
\end{enumerate}
Instead of $v.q$, we often use the notation $vq$.
\end{definition}

\begin{definition}[Right quaterninoic pre-Hilbert space]
A right quaternionic pre-Hilbert space $\mathcal{H}$, is a right quaternionic vector space equipped with the binary mapping\\ \( \langle \cdot \mid \cdot \rangle : \mathcal{H} \times \mathcal{H} \to \mathbb{H} \) (called the Hermitian quaternionic inner product) which satisfies the following properties:

\begin{itemize}
    \item[(a)] \( \langle v_1 \mid v_2 \rangle = \overline{\langle v_2 \mid v_1 \rangle} \) for all \( v_1, v_2 \in \mathcal{H}\),
    \item[(b)] \( \langle v \mid v \rangle > 0 \) if \( v \neq 0 \),
    \item[(c)] \( \langle v \mid v_1 + v_2 \rangle = \langle v \mid v_1 \rangle + \langle v \mid v_2 \rangle \) for all \( v, v_1, v_2 \in \mathcal{H} \),
    \item[(d)] \( \langle v \mid uq \rangle = \langle v \mid u \rangle q \) for all \( v, u \in \mathcal{H} \) and \( q \in \mathbb{H} \).
\end{itemize}
\end{definition}

In view of Definition  $3$ , a right pre-Hilbert space $\mathcal{H}$ also has the property:

\begin{itemize}
    \item[(i)] $ \langle vq \mid u \rangle = \overline{q} \langle v \mid u \rangle$ for all $v, u \in \mathcal{H} $ and \( q \in \mathbb{H} \).
\end{itemize}

Let \( \mathcal{H} \) be a right quaternionic pre-Hilbert space with the Hermitian inner product \( \langle \cdot \mid \cdot \rangle \). Define the quaternionic norm \( \|\cdot\| : \mathcal{H} \to \mathbb{R}^+ \) on \( \mathcal{H} \) by
\[
\|u\| = \sqrt{\langle u \mid u \rangle}, \quad u \in \mathcal{H}, 
\]
which satisfies the following properties:
\begin{enumerate}
\item $\|uq\|=\|u\|\vert q\vert$, for all $u\in \mathcal{H}$ and $q\in \mathbb{H}$,
\item $\| u+v\|\leq \|u\|+\|v\|$,
\item $\|u\|=0\Longleftrightarrow u=0$ for $u\in \mathcal{H}$.
\end{enumerate}
\begin{definition}[ Right quaternionic Hilbert space]
A right quaternionic pre-Hilbert space is called a right quaternionic Hilbert space if it is complete with respect to the quaternionic norm.
\end{definition}

\begin{example}
Define $$\ell^2(\mathbb{H}):=\left\{ \{q_i\}_{i\in I}\subset \mathbb{H}:\; \displaystyle{\sum_{i\in I}\vert q_i\vert^2<\infty} \right\}.$$
$\ell^2(\mathbb{H})$ under right multiplication by quaternionic scalars together with the quaternionic inner product defined as: $\langle p\mid q\rangle:=\displaystyle{\sum_{i\in I} \overline{p_i}q_i}$ for  $p=\{p_i\}_{i\in I}$ and $q=\{q_i\}_{i\in I}\in \ell^2(\mathbb{H})$, is a right quaternionic Hilbert space.
\end{example}
\begin{theorem}[The Cauchy-Schwarz inequality]\cite{10}
If $\mathcal{H}$ is a right quaternionic Hilbert space, then for all $u,v\in \mathcal{H}$, 
$$\vert \langle u\mid v\rangle \vert\leq \|u\| \|v\|.$$
\end{theorem}
\begin{definition}[orthogonality]
Let \(\mathcal{H} \) be a right quaternionic Hilbert space and \( A\) be a subset of \( \mathcal{H} \). Then, define the set:

\begin{itemize}
    \item \( A^{\perp} = \{ v \in \mathcal{H} : \langle v \mid u \rangle = 0 \; \forall \; u \in A \} \);
    \item \( \langle A \rangle \) as the right quaternionic vector subspace of \( \mathcal{H} \) consisting of all finite right \( \mathbb{H} \)-linear combinations of elements of \( A\).
\end{itemize}
\end{definition}
\begin{property}\cite{10}
Let \(\mathcal{H} \) be a right quaternionic Hilbert space and \( A\) be a subset of \( \mathcal{H} \). Then,
\begin{enumerate}
\item $A^{\perp}=\langle A\rangle^\perp=\overline{\langle A\rangle }^\perp=\overline{\langle A\rangle ^\perp}.$
\item $(A^\perp)^\perp=\overline{\langle A\rangle}.$
\item $\overline{A}\oplus A^\perp=\mathcal{H}.$
\end{enumerate}
\end{property}

\begin{theorem}\cite{10}
Let \( \mathcal{H} \) be a quaternionic Hilbert space and let \( N \) be a subset of \( \mathcal{H} \) such that, for \( z, z' \in N \), we have \( \langle z \mid z' \rangle = 0 \) if \( z \neq z' \) and \( \langle z \mid z \rangle = 1 \). Then, the following conditions are equivalent:

\begin{itemize}
    \item[(a)] For every \( u, v \in \mathcal{H} \), the series \( \sum_{z \in N} \langle u \mid z \rangle \langle z \mid v \rangle \) converges absolutely and
    \[
    \langle u \mid v \rangle = \sum_{z \in N} \langle u \mid z \rangle \langle z \mid v \rangle;
    \]
    
    \item[(b)] For every \( u \in \mathcal{H} \), \( \|u\|^2 = \displaystyle{\sum_{z \in N} |\langle z \mid u \rangle|^2 }\);
    
    \item[(c)] \( N^{\perp} = \{0\} \);
    
    \item[(d)] \( \langle N \rangle \) is dense in \( \mathcal{H} \).
\end{itemize}
\end{theorem}
\begin{definition}
A subset $N$ of $\mathcal{H}$ that satisfies one of the statements in Theorem $2$ is called Hilbert basis or orthonormal basis for $\mathcal{H}$.
\end{definition}
\begin{theorem}\cite{10}
Every quaternionic Hilbert space has a Hilbert basis.\\
\end{theorem}

\begin{definition}[Frames]\cite{11}
Let $\{u_i\}_{i\in I}$ be a sequence in a right quaternionic Hilbert space. $\{u_i\}_{i\in I}$ is said to be Frame for $\mathcal{H}$ if there exist $0<A\leq B<\infty$ such that for all $u\in \mathcal{H}$, the following inequality holds:
$$A\|u\|^2\leq \displaystyle{\sum_{i\in I}\vert \langle u_i,u\rangle \vert^2}\leq B\|u\|^2.$$
\begin{enumerate}
\item If only the upper inequality holds, $\{u_i\}_{i\in I}$ is called a Bessel sequence for $\mathcal{H}$.
\item If $A=B=1$, $\{u_i\}_{i\in I}$ is called a Parseval frame for $\mathcal{H}$.
\end{enumerate}
\end{definition}

\section{Auxiliary results}
In this section, we will present some interesting results on operator theory in quaternionic Hilbert spaces, which will be utilized in our study. The properties of the associated operators of a frame will  also be provided.\\

\begin{definition}[Right $\mathbb{H}$-linear operator]\cite{10}
Let $\mathcal{H}$ and $\mathcal{K}$ be two right quaternionic Hilbert spaces. Let  $L:\mathcal{H}\rightarrow \mathcal{K}$ be a map.
\begin{enumerate}
\item  $L$ is said to be right $\mathbb{H}$-linear operator if $L(uq+vp)=L(u)q+L(v)p$ for all $u,v\in \mathcal{H}$ and $p,q\in \mathcal{H}$.
\item  If $L$ is  a right $\mathbb{H}$-linear operator. $L$ is continuous if and only if $L$ is bounded; i.e., there exists $M> 0$ such that for all $u\in \mathcal{H}$, 
$$\| Lu\|\leq M\| u\|.$$
We denote $\mathbb{B}(\mathcal{H},\mathcal{K})$ the set of all right $\mathbb{H}$-linear bounded operators from $\mathcal{H}$ to $\mathcal{K}$, and if $\mathcal{H}=\mathcal{K}$, we denote $\mathbb{B}(\mathcal{H})$ instead of $\mathbb{B}(\mathcal{H},\mathcal{H})$.
\item If $L$ is a right $\mathbb{H}$-linear  bounded operator, we define the norm of $L$ as: 
$$\|L\|=\displaystyle{\sup_{\|u\|=1}\|Lu\|}=\inf\{M>0:\; \| Lu\|\leq M\| u\|,\; \forall u\in \mathcal{H}\}.$$
And we have for all $L,M\in \mathbb{B}(\mathcal{H}), \|L+M\|\leq\|L\|+\|M\|$ and $\|MN\|\leq\|L\|\|M\|.$\\
\end{enumerate}
\end{definition}
\begin{theorem}[Quaternionic representation Riesz' theorem]\cite{10}
    If $\mathcal{H}$ is a right  quaternionic Hilbert space, the map
    \[
          v\in \mathcal{H} \mapsto \langle v \mid \cdot \rangle \in \mathcal{H}'
    \]
    is well-posed and defines a conjugate-$\mathbb{H}$-linear isomorphism.\\
\end{theorem}
\begin{theorem}[ The uniform boundedness principle]\cite{10}
Let $\mathcal{H}$ be a right quaternionic Hilbert space and $F$ be any subset $F$ of $\mathbb{B}(\mathcal{H})$, if
\[
    \sup_{L \in F} \|L v\| < +\infty \quad \text{for every } v \in \mathcal{H},
\]
then:
\[
    \sup_{L \in F} \|L\| < +\infty.
\]
\end{theorem}
\begin{theorem}[ The open map theorem]\label{thm6}\cite{10}
Let $\mathcal{H}$ be a right quaternionic Hilbert space. If \( L \in \mathbb{B}(\mathcal{H}) \) is surjective, then \( L \) is open. In particular, if \( L \) is bijective, then \( L^{-1} \in \mathbb{B}(\mathcal{H}). \)\\
\end{theorem}
\begin{theorem}[ The closed graph theorem]\label{thm7}\cite{10}
Let $\mathcal{H}$ be a right quaternionic Hilbert space and let $L:\mathcal{H}\rightarrow \mathcal{H}$ be a right $\mathbb{H}$-linear opeartor. If $Graph(L)$ is closed, then $L\in \mathbb{B}(\mathcal{H})$.\\
\end{theorem}

\begin{definition}[the adjoint operator]\cite{10}
Let $\mathcal{H}$ be a right quaternionic Hilbert space and  $L\in \mathbb{B}(\mathcal{H})$. The adjoint operator of $L$, denoted $L^*$, is the unique operator in $\mathbb{B}(\mathcal{H})$ satisfying for all $u,v\in \mathcal{H}$:
$$\langle Lu\mid v\rangle=\langle u\mid L^*v\rangle.$$
\end{definition}
\begin{definition}\cite{10}
Let $\mathcal{H}$ be a right quaternionic Hilbert space and  $L\in \mathbb{B}(\mathcal{H})$.
\begin{enumerate}
\item $L$ is called normal if $LL^*=L^*L$.
\item $L$ is called self-adjoint if $L=L^*$.
\item $L$ is called isometric if $\| Lu\|=\|u\|$ for all $u\in \mathcal{H}$.
\item $L$ is called unitary if $LL^*=L^*L=I$, where $I$ is the identity opertaor of $\mathbb{B}(\mathcal{H})$. An operator is a unitary if and only if it is an isometric surjective operator. 
\item $L$ is called positive, and we write $L\geqslant 0$, if $\langle Lu\mid u\rangle\geqslant 0$ for all $u\in \mathcal{H}$.\\
\end{enumerate}
\end{definition}

\begin{proposition}\label{prop1}\cite{10}
Let $\mathcal{H}$ be a right quaternionic Hilbert space and  $L\in \mathbb{B}(\mathcal{H})$. Then:
\begin{enumerate}
\item $R(L)^\perp=ker(L^*)$.
\item $\overline{R(L^*)}=ker(L)^\perp.$\\
\end{enumerate}
\end{proposition}
\begin{theorem}[Square root of an operator]\cite{10}
    Let \( \mathcal{H} \) be a right quaternionic Hilbert space and let \( L \in \mathbb{B}(\mathcal{H}) \). If \( L \geq 0 \), then there exists a unique operator in \( \mathbb{B}(\mathcal{H})\), indicated by \( \sqrt{L} \), such that \( \sqrt{L} \geq 0 \) and 
\[
\sqrt{L} \sqrt{L} = L.
\]
Furthermore, it turns out that \( \sqrt{L} \) commutes with every operator that commutes with \( L \).\\
\end{theorem}

\begin{theorem}[ The closed range theorem]\label{thm9}\cite{10}
 Let \( \mathcal{H} \) be a right quaternionic Hilbert space and let \( L \in \mathbb{B}(\mathcal{H}) \). If $R(L)$ is closed, then $R(L^*)$ is also closed.\\
\end{theorem}

\begin{proposition}\label{prop2}
 Let \( \mathcal{H} \) be a right quaternionic Hilbert space and let \( L \in \mathbb{B}(\mathcal{H}) \). If $L$ is bounded below, i.e. there exists $M> 0$ such that for all $u\in \mathcal{H}$, $M\|u\|\leq \| Lu\|$, and $L^*$ is injective, then $L$ is invertible.
\end{proposition}
\begin{proof}
Since $L$ is bounded below, then it is injective. Assume that \( L^* \) is injective, so \( \operatorname{Ker}(L^*) = \{0\} \), hence $\overline{R(L)} = \mathcal{H}$. Let us show that $R(L)$ is closed in \( \mathcal{H} \). For this,  let $y \in \overline{R(L)}$, so there exists a sequence $\{y_n\}_{n \geq 1}$ in $ R(L)$ such that $\displaystyle{ \lim_{n \to \infty} y_n = y}$. For each $n \geq 1$, there exists $x_n \in \mathcal{H}$ such that $ y_n = L(x_n)$. We have:
\[
m \| x_n - x_m \| \leq \| L(x_n) - L(x_m) \| = \| y_n - y_m \|, \quad \forall n, m \geq 1.
\]
This implies that $\{x_n\}_{n \geq 1}$ is a Cauchy sequence in $\mathcal{H}$, hence convergent, and let $x$ be its limit. Since $L$ is bounded, we have $y = \displaystyle{\lim_{n \to \infty} y_n = \lim_{n \to \infty} L(x_n) = L(x)}$. Thus, $y \in R(L)$. Therefore, $R(L)$ is closed, so $ R(L) = \mathcal{H}$. Thus, $ L$ is surjective and therefore invertible.
\end{proof}
\begin{proposition}\label{prop3}
 Let $ \mathcal{H}$ be a right quaternionic Hilbert space and let \( L \in \mathbb{B}(\mathcal{H}) \) be normal. Then, the following statements are equivalent:
 \begin{enumerate}
 \item $L$ is invertible.
 \item $L$ is bounded below.
 \end{enumerate}
\end{proposition}
\begin{proof}
It is clear that $1.$ implies $2.$ by taking $M = \displaystyle{\frac{1}{\| L^{-1} \|}}$ as the lower bound. Conversely, 
we use the fact that $\| Lu\| = \| L^*u \|$, $\forall u \in \mathcal{H}$. Then, we apply Proposition \ref{prop2}.
\end{proof}
\begin{proposition}\label{prop4}
 Let \( \mathcal{H},\mathcal{K} \) be two right quaternionic Hilbert spaces and let \( L \in \mathbb{B}(\mathcal{K},\mathcal{H}) \). Then, the following statements are equivalent:
 \begin{enumerate}
 \item $L$ is injective with closed range.
 \item $L$ is bounded below.
 \item $L^*$ is surjective.
 \end{enumerate}
\end{proposition}
\begin{proof}
Proposition \ref{prop1} and Theorem \ref{thm9} together prove the equivalence $1.\Longleftrightarrow 3.$ Assume that $L$ is injective with closed range, then the operator $L : \mathcal{H} \to R(L)$ is invertible, and since $R(L)$ is closed, it follows, by the open map theorem \ref{thm6}, that $L^{-1} : R(L) \to \mathcal{H}$ is bounded. Now, assume for contradiction that $L$ is not bounded below. Then, for every $n \geq 1$, there exists $u_n \in \mathcal{K}$ with $\| u_n \| = 1$ such that $\displaystyle{\frac{1}{n}} \geq \| L(u_n) \|$, hence $Lu_n\to 0$, which implies $u_n = L^{-1}(L(u_n)) \to 0$. This is a contradiction since $\| u_n \| = 1$, $\forall n \geq 1$. Conversely,  assume that $L$ is bounded below, thus, there exists a strictly positive constant $M$ such that for all $x \in \mathcal{K}$, $M \| x \| \leq \| Lx\|$. It is clear that $L$ is injective. Let $\{x_n\}_{n \geq 1} \in \mathcal{K}$ be such that $Lx_n\to y \in \mathcal{H}$ and let's  show that $y \in R(L)$. We have 

\[
\alpha \| x_n - x_m \| \leq \| L(x_n - x_m) \| = \| L(x_n) - L(x_m) \|,
\]

thus $\{x_n\}_{n \geq 1}$ is a Cauchy sequence in $\mathcal{K}$, which means it converges to some $x \in \mathcal{K}$. Since $L$ is bounded, $Lx_n \to L(x)$, and by the uniqueness of limits, we conclude that $y = L(x)$ and therefore $y \in R(L)$. Hence, $R(L)$ is closed.
\end{proof}

Let $L \in \mathbb{B}(\mathcal{K},\mathcal{H})$ be a right $\mathbb{H}$-linear bounded operator with closed range, and let $R_L = \operatorname{Im}(L)$
and $N_L = \operatorname{Ker}(L)$. The restriction of $L$ on $N_L^\perp$, denoted $L_{|N_L^\perp}$, is injective.
Indeed, if $x \in N_L^\perp$, we have
\[
L_{|N_L^\perp}(x) = 0 \implies x \in N_L^\perp \cap N_L = 0.
\]
On the other hand, we have $L_{|N_L^\perp}(N_L^\perp) = L(\mathcal{K}) = R_L$. Therefore:
\[
L_{|N_L^\perp} : N_L^\perp \to R_L
\]
is invertible. And since $N_L^\perp$ and $R_L$ are closed, it follows, by the open map theorem \ref{thm6}, that:
\[
(L_{|N_L^\perp})^{-1} : R_L \to N_L^\perp
\]
is bounded.
\begin{definition}[Pseudo-inverse of an operator with closed range]
Let $\mathcal{H}$ and $\mathcal{K}$ be two right quaternionic Hilbert spaces and 
\( L \in \mathbb{B}(\mathcal{K},\mathcal{H}) \) be with closed range $R_L$ and  kernel \( N_L \). The pseudo-inverse of \( L \), denoted by \( L^\dagger \), is the right $\mathbb{H}$-linear bounded operator extending 
\(
(L _{|N_L^\perp})^{-1} : R_L \to N_L^\perp
\)
to \( \mathcal{H} \) by the property \( ker(L^\dagger) = R(L^\perp) \), i.e., 
\[
L^\dagger = (L _{|N_L^\perp})^{-1} P_{R_L},
\]
where $P_{R_L}$ is orthogonoal projection onto $R_L$.\\
\end{definition}

\begin{theorem}\label{thm10}
Let \( L \in \mathbb{B}(\mathcal{K},\mathcal{H}) \) be  with closed range \( R_L \), and let \( v \in R_L \). The equation
\[
Lx = v 
\]
admits a unique solution of minimal norm. This solution is exactly \( L^\dagger(v) \).
\end{theorem}
\begin{proof}
Let us first show that \( L^\dagger(v) \) is a solution. We have \( v \in R_L \), thus \( L^\dagger(v) = (L _{|N_L^\perp})^{-1}(v) \), which gives us \( L(L^\dagger(v)) = v \). Let \( g \in \mathcal{K} \) be a solution to this equation. There exist \( g_1 \in N_L \) and \( g_2 \in N_L^\perp \) such that \( g = g_1 + g_2 \).  Let us show that \( g_2 = L^\dagger(v) \): We have \( L(g) = L(g_2) \), and since \( g \) is a solution to the equation, we get \( L(g_2) = v \). Because \( g_2 \) and \( L^\dagger(v) \) belong to \( N_L^\perp \) and \( L_{|N_L^\perp} \) is invertible, it follows that \( g_2 = L^\dagger(v) \). Thus, we have \( g = g_1 + L^\dagger(v) \). Consequently, 
\[
\|g\|^2 = \|g_1\|^2 + \|L^\dagger(v)\|^2,
\]
which implies that: 
\[
\|g\| \geq \|L^\dagger(v)\|.
\]

Furthermore, we have:
\[
\|g\| = \|L^\dagger(v)\| \iff \|g_1\| = 0 \iff g_1 = 0 \iff g = L^\dagger(v).
\]
\end{proof}

\begin{definition}\cite{11}
Let $\{u_i\}_{i\in I}$ be a Bessel sequence for $\mathcal{H}$. 
\begin{enumerate}
\item The pre-frame operator of $\{u_i\}_{i\in I}$ is the right $\mathbb{H}$-linear bounded operator denoted by $T$ and  defined as follows: $$\begin{array}{rcl}
T:\ell^2(\mathbb{H})&\rightarrow& \mathcal{H}\\
q:=\{q_i\}_{i\in I}&\mapsto& \displaystyle{\sum_{i\in I}u_iq_i}. 
\end{array}$$
\item The transform opeartor of $\{u_i\}_{i\in I}$, denoted by $\theta$, is the adjoint of its pre-frame operator. explicitly $\theta$ is defined as follows:
$$\begin{array}{rcl}
\theta: \mathcal{H}&\rightarrow& \ell^2(\mathbb{H})\\
u&\mapsto&\{\langle u_i,u\rangle\}_{i\in I}.
\end{array}$$
\item The frame operator of $\{u_i\}_{i\in I}$, denoted by $S$, is the composite of $T$ and $\theta$. explicitly, $S$ is defined as follows: 
$$\begin{array}{rcl}
S:\mathcal{H}&\rightarrow& \mathcal{H}\\
u&\mapsto&\displaystyle{\sum_{i\in I}u_i\langle u_i,u\rangle}.
\end{array}$$
\end{enumerate}
\end{definition}
\begin{proposition}\cite{11}
Let \( (u_i)_{i \in I} \) be a frame of \( \mathcal{H} \). Then:
\begin{enumerate}
    \item \( T^* \) is a right $\mathbb{H}$-linear  bounded injective operator.
    \item \( T \) is a right $\mathbb{H}$-linear bounded surjective operator.
    \item \( S \) is a right $\mathbb{H}$-linear bounded, positive, and invertible operator.
\end{enumerate}
\end{proposition}
\begin{remark}
Let $\{u_i\}_{i\in I}$ be a Bessel sequence for $\mathcal{H}$. Then, for all $u\in \mathcal{H}$, we have: $$\displaystyle{\sum_{i\in I}\vert \langle u_i,u\rangle \vert^2}=\langle Su,u\rangle=\|T^*u\|^2=\|\theta u\|^2.$$
\end{remark}
\section{Frame coefficients}
In all this section $\mathcal{H}$ is a right quaternionic Hilbert space, $\mathbb{B}(\mathcal{H})$ is the set of all right $\mathbb{H}$-linear bounded operators on $\mathcal{H}$ and $I$ is a countable set. In what follows, we show that any vector in a right quaternionic Hilbert space can be expressed (in a generally non-unique way) in terms of the elements of a given frame and that it admits a natural representation with an interesting characteristic.\\

\begin{proposition}
Let $(u_i)_{i\in I}$ be a frame of $\mathcal{H}$. Then:
$$ \forall u \in \mathcal{H}, \, u = \sum_{i\in I} u_i\langle u_i,S^{-1}u \rangle  = \sum_{i\in I} S^{-1}u_i\langle u_i, u\rangle. $$
\begin{enumerate}
\item[]$\rightarrow$ The coefficients $\{\langle u_i,S^{-1}u\}_{i\in I}$ are called the frame coefficients for $u$. 
\item[] $\rightarrow$ The expression $ u = \displaystyle{\sum_{i\in I} u_i\langle u_i,S^{-1}u \rangle}$  is often called the natural representation of $u$.
\end{enumerate}
\end{proposition}
\begin{proof}
For $u\in \mathcal{H}$, we have: $u=SS^{-1}u=\displaystyle{\sum_{i\in I} u_i\langle u_i,S^{-1}u \rangle}$, and on the other hand, we have, $u=S^{-1}Su=S^{-1}\left( \displaystyle{\sum_{i\in I}u_i\langle u_i,u\rangle  } \right)=\displaystyle{\sum_{i\in I} S^{-1}u_i\langle u_i,u \rangle}.$
\end{proof}
\begin{remark}
In general, the decomposition of a vector in a right quaternionic Hilbert space with respect to a frame is not unique, as shown in the following example: Let $(v_i)_{i \geq 1}$ be a Hilbert basis of $\mathcal{H}$. Then the sequence $(v_1, v_1, v_2, v_3, v_4, \ldots)$ is a frame of $\mathcal{H}$. If we set $u = 2v_1$, we can see that $u$ can also be expressed as $u = v_1 + v_1$, which means that $(1, 1, 0, 0, \ldots)$ and $(2, 0, 0, \ldots)$ are two different representations of the same vector $u$ in the frame $(v_1, v_1, v_2, v_3, v_4, \ldots)$.\\
\end{remark}

The following theorem shows the particularity of the frame coefficients.
\begin{theorem}\label{thm11}
Let $\{u_i\}_{i\in I}$ be a frame for $\mathcal{H}$ and let $u\in \mathcal{H}$ such that $u=\displaystyle{\sum_{i\in I}u_iq_i}$ where $\{q_i\}_{i\in I}\in \ell^2(\mathbb{H})$. Then: 
$$\sum_{i\in I}\vert q_i\vert^2=\sum_{i\in I}\langle u_i,S^{-1}u\rangle \vert^2+\sum_{i\in I}\vert \langle u_i,S^{-1}u\rangle-q_i\vert^2.$$
In particular, the frame coefficients $\{\langle u_i,S^{-1}u\rangle \}_{i\in I}$ of $u$ with respect to the frame $\{u_i\}_{i\in I}$ represent the representation with minimal $\ell^2(\mathbb{H})$-norm."
\end{theorem}
\begin{proof}
We have $\displaystyle{\sum_{i\in I}u_i\langle u_i,S^{-1}u\rangle=\sum_{i\in I}u_iq_i}$, then, by multiplying both terms by $S^{-1}u$ on the left (in the sense of the inner product), we obtain:
$$\sum_{i\in I}\vert \langle u_i,S^{-1}u\rangle\vert^2=\sum_{i\in I}\overline{\langle u_i,S^{-1}u\rangle} q_i.$$
On the other hand, we have:
$$\sum_{i\in I}\vert \langle u_i,S^{-1}u\rangle-q_i\vert^2=\sum_{i\in I}\vert \langle u_i,S^{-1}u\rangle\vert^2-2\text{Re}(\sum_{i\in I}\overline{\langle u_i,S^{-1}u\rangle} q_i)+\sum_{i\in I}\vert q_i\vert^2.$$
\end{proof}
\begin{lemma}
Let $\{u_i\}_{i\in I}$ be a frame of $\mathcal{H}$. Then for any $u\in \mathcal{H}$, we have:
\[
T^\dagger(u) = \left( \langle u_i, S^{-1}u \rangle \right)_{i\in I}.
\]
That is:
\[
T^\dagger = T^* S^{-1}.
\]
\end{lemma}
\begin{proof}
$T$ is a surjective right $\mathbb{H}$-linear bounded operator, then $R(T)=\mathcal{H}$. Let $u\in \mathcal{H}$ and  consider the following equation $(E): Tx=u$. By Theorem \ref{thm10}, The equation $(E)$ has a unique solution with minimal norm which is exactely $T^{\dagger}u$. On the other hand, $q=\{q_i\}_{i\in I}\in \ell^2(\mathbb{H})$ is a solution to $(E)$, if and only if, $\displaystyle{\sum_{i\in I}u_iq_i}=u$, then the unique solution with minimal norm of $(E)$ is, by Theorem \ref{thm11}, $\{\langle u_i,S^{-1}u\rangle\}_{i\in I}$. Hence $T^{\dagger}u=\{\langle u_i,S^{-1}u\rangle\}_{i\in I}=T^*S^{-1}u$.

\end{proof}
By definition of a frame, the optimal frame bounds are given by:
\[
A_{opt} := \inf_{\|f\|=1} \langle Sf, f \rangle, \quad B_{opt} := \sup_{\|f\|=1} \langle Sf, f \rangle.
\]

The following theorem express the optimal frame bounds of a frame using its assiciated operators.
\begin{theorem}\label{thm12}
Let $\{u_i\}_{i\in I}\subset \mathcal{H}$ be a frame, $T$ and $S$ be, respectively, its pre-frame operator  and its frame operator. let $A_{opt}\leq B_{opt}$ be the optimal frame bounds of $\{u_i\}_{i\in I}$. Then: 
\begin{enumerate}
\item $A_{opt}=\displaystyle{\frac{1}{\|S^{-1}\|}=\frac{1}{\|T^{\dagger}\|^2}}.$
\item $B_{opt}=\|S\|=\|T\|^2.$
\end{enumerate}
\end{theorem}
\begin{proof}
We have $\langle Sf,f\rangle=\langle S^{\frac{1}{2}}f,S^{\frac{1}{2}}f\rangle=\|S^{\frac{1}{2}}f\|^2$, then $B_{opt}=\|S^{\frac{1}{2}}\|^2=\|S\|$ and $A_{opt}=\displaystyle{\frac{1}{\|S^{\frac{-1}{2}}\|^2}}=\displaystyle{\frac{1}{\|S^{-1}\|}}.$ On the other hand, we have $\|S\|=\|TT^*\|=\|T\|^2$ and since $T^{\dagger}=T^*S^{-1}$, then $\|T^{\dagger}\|^2=\|T^{\dagger^*}T^{\dagger}\|=\|S^{-1}TT^*S^{-1}\|=\|S^{-1}SS^{-1}\|=\|S^{-1}\|.$
\end{proof}

Let $\{u_i\}_{i\in I}$ be a frame for $\mathcal{H}$ with the optimal frame bounds $A\leq B$.  If $R\in \mathbb{B}(\mathcal{H})$ and $u\in \mathcal{H}$, the frame coefficients of $Ru$ are $\{\langle u_i,S^{-1}Ru\rangle\}_{i\in I}$ . An interesting question arises: Can we determine the frame coefficients of 
$Ru$ from those of $u$?

We consider the map defined as follows: 
$$\begin{array}{rcl}
\Lambda:\ell^2(\mathbb{H})&\rightarrow& \ell^2(\mathbb{H})\\
\{q_i\}_{i\in I}&\mapsto& \left\{\displaystyle{\sum_{i\in I}\langle S^{-1}u_n,Ru_i\rangle q_i}\right\}_{n\in I}.
\end{array}$$
\begin{proposition}
$\Lambda$ is a well defined right $\mathbb{H}$-linear bounded operator. i.e., $\Lambda\in \mathbb{B}(\mathcal{H})$.
\end{proposition}
\begin{proof}
Let $\{q_i\}_{i\in I}\in \ell^2(\mathbb{H})$. Il is clear that, for all $n\in I$, $\displaystyle{\sum_{i\in I} \langle S^{-1}u_n,Ru_i\rangle q_i}\in \mathbb{H}$ since $\{\langle S^{-1}u_n,Ru_i\rangle\}_{i\in I},\; \{q_i\}_{i\in I}\in \ell^2(\mathbb{H})$.
Let $J\subset I$ be a finite subset of $I$. Then:
$$\begin{array}{rcl}
\displaystyle{\sum_{n\in J}\vert \sum_{i\in I}\langle S^{-1}u_n,Ru_i\rangle q_i\vert^2}&=&\displaystyle{\sum_{n\in J}\vert \langle S^{-1}u_n,RT(\{q_i\}_{i\in I}\rangle \vert^2}\\
&=&\displaystyle{\sum_{n\in J}\vert \langle u_n,S^{-1}RT(\{q_i\}_{i\in I})\rangle \vert^2}\\
&\leq& B\|S^{-1}RT(\{q_i\}_{i\in I})\|^2\leq B\|S^{-1}\|^2\|R\|^2\|T\|^2\|\{q_i\}_{i\in I}\|\\
&=&\left(\displaystyle{\frac{B\|R\|}{A}}\right)^2\, \left\|\{q_i\}_{i\in I}\right\|^2.
\end{array}$$
Hence, $\Lambda$ is well defined, clearly right $\mathbb{H}$-linear and bounded operator, moreover $\|\Lambda\|\leq \displaystyle{\frac{B\|R\|}{A}}$.\\
\end{proof}

The answer to the above question is given in the following proposition.
\begin{proposition}
For all $u\in \mathcal{H}$, the frame coefficients of $Ru$ are obtained from those of $u$ via
the right $\mathbb{H}$-linear bounded operator $\Lambda$.
\end{proposition}
\begin{proof}
For $u\in \mathcal{H}$, we have: 
$$\begin{array}{rcl}
\Lambda\left( \{\langle u_i,S^{-1}u\rangle \}_{i\in I}\right)&=& \left( \displaystyle{\sum_{i\in I }\langle S^{-1}u_n,Ru_i\rangle \langle u_i,S^{-1}u\rangle}\right)_{n\in I}\\
&=&\left( \left\langle S^{-1}u_n,R(\displaystyle{\sum_{i\in I}u_i\langle u_i,S^{-1}u\rangle)}\right\rangle\right)_{n\in I}\\
&=&\left(\langle S^{-1}u_n,Ru\rangle\right)_{n\in I}\\
&=&\left(\langle u_n,S^{-1}Ru\rangle\right)_{n\in I}.
\end{array}$$
\end{proof}
\section{Frames and operators on quaternionic Hilbert spaces}
In all this section $\mathcal{H}$ is a right quaternionic Hilbert space, $\mathbb{B}(\mathcal{H})$ is the set of all right $\mathbb{H}$-linear bounded operators on $\mathcal{H}$ and $I$ is a countable set. In what follows, we characterize frames in a right quaternionic Hilbert space  by the associated operators. The frames of the form $\{Lu_i\}_{i\in I}$, where $L\in \mathbb{B}(\mathcal{H})$ and $\{u_i\}_{i\in I}$ is a frame for $\mathcal{H}$, are studied.\\

We first give a characterization of frames by the pre-frame operator.
\begin{theorem}
Let $\{u_i\}_{i\in I}\subset \mathcal{H}$ and $T$ be its pre-frame operator. Then, the following statements are equivalent:
\begin{enumerate}
\item $\{u_i\}_{i\in I}$ is a frame for $\mathcal{H}$.
\item $T$ is well defined, bounded and surjective.
\end{enumerate}
\end{theorem}
\begin{proof}
We have already seen that $1.$ implies $2.$ Conversely, since $T$ is surjective, then, by Proposition \ref{prop4}, $T^*$ ( which is well defined and bounded since $T$ is bounded) is bounded below. The fact that $\displaystyle{\sum_{i\in I}\vert \langle u_i,u\rangle\vert^2}=\| T^*u\|^2$ completes the proof.
\end{proof}

Now, we characterize frames by the frame operator.
\begin{theorem}
Let $\{u_i\}_{i\in I}\subset \mathcal{H}$ and $S$ be its frame operator. Then, the following statements are equivalent:
\begin{enumerate}
\item $\{u_i\}_{i\in I}$ is a frame for $\mathcal{H}$.
\item $S$ is well defined and surjective.
\end{enumerate}
\end{theorem}
\begin{proof}
It is well known that $1.$ implies $2.$. Assume that: 
$$\begin{array}{rcl}
S:\mathcal{H}&\rightarrow& \mathcal{H}\\
u&\mapsto& \displaystyle{\sum_{i\in I}u_i\langle u_i,u\rangle},
\end{array}$$ is well defined and surjective.
Let $\{w_n\}_{n\in \mathbb{N}}\subset \mathcal{H}$ such that $w_n \to w\in \mathcal{H}$ and $Sw_n\to w'\in \mathcal{H}$. Since $S$ is surjective, then there exists $v\in \mathcal{H}$ such that $w'=Sv$. Then:
$$\begin{array}{rcl}
Sw_n\to Sv \text{ and } w_n \to w &\Longrightarrow& \langle S(w_n-v),w_n-v\rangle\to 0\; (\text{ Cauchy-Schwarz inequality})\\
&\Longrightarrow& \displaystyle{\sum_{i\in I}\vert \langle u_i,w_n-v\rangle \vert^2}\to 0 \;(\text{ by definition of } S)\\
&\Longrightarrow& \langle u_i,w_n-v\rangle \to 0\;\, (\forall i\in I)\\
&\Longrightarrow& \langle u_i,w_n\rangle \to \langle u_i,v\rangle\\
&\Longrightarrow& \langle u_i,w\rangle =\langle u_i,v\rangle \;\,(\forall i\in I)\\
&\Longrightarrow& \displaystyle{\sum_{i\in I}u_i\langle u_i,w\rangle=\sum_{i\in I}u_i\langle u_i,v\rangle}\\
&\Longrightarrow& S(w)=S(v)=w'.
\end{array}$$
Then, the graph of $S$ is closed, hence, by the closed graph theorem \ref{thm7}, $S$ is bounded. Since $S$ is positive and surjective, then it is invertible. Thus, $S^{\frac{1}{2}}$ is positive and invertible. Hence the existence of $0< A\leq B$ such that for all $u\in \mathcal{H}$, $A\|u\|^2\leq\|S^{\frac{1}{2}}u\|^2\leq B\|u\|^2$. The fact that $\|S^{\frac{1}{2}}u\|^2=\langle Su,u\rangle=\displaystyle{\sum_{i\in I}\vert \langle u_i,u\rangle\vert^2}$ completes the proof.
\end{proof}

Let $L$ be a right $\mathbb{H}$-linear bounded operator and $\{u_i\}_{i\in I}$ be a frame for $\mathcal{H}$. In what follows, we study the sequence $\{Lu_i\}_{i\in I}$.

In general, the image of a frame under a bounded operator (even if it is injective) is not a frame, as illustrated by the following example: Consider \(\{v_i\}_{i\geqslant 1}\), a Hilbert basis of \(\mathcal{H}\), and the following injective bounded linear operator:
$$\begin{array}{rcl}
L : \mathcal{H} &\rightarrow& \mathcal{H}\\

x &\mapsto& \displaystyle{\sum_{i=1}^{+\infty} v_{i+1}\langle v_i,x \rangle}.
\end{array}$$
We have $\{L(v_i)\}_{i\geqslant 1}= \{v_i\}_{i\geqslant2}$, which is not even complete.\\

The following proposition expresses the frame operator of $\{Lu_i\}_{i\in I}$ whenever it forms a frame (or only a Bessel sequence).
\begin{proposition}\label{prop8}
Let $L\in \mathbb{B}(\mathcal{H})$ and $\{u_i\}_{i\in I}$ be a frame for $\mathcal{H}$.If $\{Lu_i\}_{i\in I}$ is a frame, then its frame operator is $LSL^*$.
\end{proposition}
\begin{proof}
Denote by $S_L$ the frame operator for $\{Lu_i\}_{i\in I}$. For all $u\in \mathcal{H}$, we have: 
$$\begin{array}{rcl}
\langle S_Lu,u\rangle=\displaystyle{\sum_{i\in I}\vert \langle Lu_i,u\rangle\vert^2}&=&\displaystyle{\sum_{i\in I}\vert \langle u_i,L^*u\rangle\vert^2}\\
&=&\langle SL^*u,L^*u\rangle\\
&=&\langle LSL^*u,u\rangle.
\end{array}$$
Hence, $S_L=LSL^*$.
\end{proof}
The next theorem presents a  necessary and sufficient condition on $L\in \mathbb{B}(\mathcal{H})$ for it to transform one frame into another.
\begin{theorem}\label{thm15}
Let $L\in \mathbb{B}(\mathcal{H})$ and $\{u_i\}_{i\in I}$ be a frame for $\mathcal{H}$ with frame bounds $A\leq B$. Then,  $\{Lu_i\}_{i\in I}$ is a frame for $\mathcal{H}$ if and only if $L$ is surjective.
\end{theorem}
\begin{proof}
Assume that $\{Lu_i\}_{i\in I}$ is a frame for $\mathcal{H}$ and let $v\in \mathcal{H}$. Then, $v=\displaystyle{\sum_{i\in I} Lu_i\langle (LSL^*)^{-1}u_i,u\rangle= L\left(\sum_{i\in I}u_i\langle (LSL^*)^{-1}u_i,u\rangle\right)}$. Set $u:=\displaystyle{\sum_{i\in I}u_i\langle (LSL^*)^{-1}u_i,u\rangle \in \mathcal{H}}$, then $v=Lu$. Hence, $L$ is surjective. Conversely, 
by Proposition \ref{prop4}, $L$ is surjective if and only if $L^*$ is bounded below. Then there exists $M>0$ such that for all $u\in \mathcal{H}$,  $M\|u\|\leq \|L^*u\|$. Let $u\in \mathcal{H}$, we have:
$$
\displaystyle{\sum_{i\in I}\vert\langle Lu_i,u\rangle\vert^2}=\displaystyle{\sum_{i\in I}\vert \langle u_i,L^*u\rangle\vert^2}.
$$
Then, $$A\|L^*u\|^2\leq \displaystyle{\sum_{i\in I}\vert \langle Lu_i,u\rangle\vert^2}\leq B\|L^*u\|^2.$$ 
Thus, $$AM^2\|u\|^2\leq \displaystyle{\sum_{i\in I}\vert \langle Lu_i,u\rangle\vert^2}\leq B\|L\|^2\|u\|^2.$$
Hence, $\{Lu_i\}_{i\in I}$ is a frame for $\mathcal{H}$ with frame bounds $AM^2$ and $B\|L\|^2.$
\end{proof}
\begin{corollary}\label{cor1}
Let $\{u_i\}_{i\in I}$ be a frame for $\mathcal{H}$ with frame bounds $A\leq B$ and let $S$ be its frame operator. Then:
\begin{enumerate}
\item $\{S^{-1}u_i\}_{i\in I}$ is a frame for $\mathcal{H}$ with frame bounds $\displaystyle{\frac{1}{B}}$ and $\displaystyle{\frac{1}{A}}$. $\{S^{-1}u_i\}_{i\in I}$ is called the canonical dual frame for $\{u_i\}_{i\in I}$.
\item $\{S^{\frac{-1}{2}}u_i\}_{i\in I}$ is a Parseval frame for $\mathcal{H}$.
\end{enumerate}
\end{corollary}
\begin{proof}
\begin{enumerate}
\item By Proposition \ref{prop8},  Theorem \ref{thm12} and Theorem \ref{thm15}, we deduce easily that $\{S^{-1}u_i\}_{i\in I}$ is a frame with the optimal frame bounds $\displaystyle{\frac{1}{B_{opt}}}$ and $\displaystyle{\frac{1}{A_{opt}}}$ where $A_{opt}$ and $B_{opt}$ are the optimal frame bounds of $\{u_i\}_{i\in I}$. Next, we can deduce the same result for arbitrary frame bounds.
\item In view of Proposition \ref{prop8}, Theorem \ref{thm12} and Theorem \ref{thm15} again, we deduce that $\{S^{\frac{-1}{2}}u_i\}_{i\in I}$ is a frame for $\mathcal{H}$ with the optimal frame bounds $1$.
\end{enumerate}
\end{proof}
The following result shows that unitary right $\mathbb{H}$-linear bounded operators transform frames to other frames with the same frame bounds.
\begin{proposition}
Let $\{u_i\}_{i\in I}$ be a frame for $\mathcal{H}$, $S$ be its frame operator and let  $U\in \mathbb{B}(\mathcal{H})$ be a unitary right $\mathbb{H}$-linear bounded operator. Then $\{Uu_i\}_{i\in I}$ is a frame for $\mathcal{H}$ with the same frame bounds.
\end{proposition}
\begin{proof}
By Theorem \ref{thm15}, $\{Uu_i\}_{i\in I}$ is a frame for $\mathcal{H}$.  It sufficies to show that $\{u_i\}_{i\in I}$ and $\{Uu_i\}_{i\in I}$ have the same optimal frame bounds. i.e. in view of Theorem \ref{thm12} and Proposition \ref{prop8}, we will show that $\|USU^*\|=\|S\|$ and $\|(USU^*)^{-1}\|=\|S^{-1}\|.$ We have: $$\begin{array}{rcl}
\|USU^*\| &=& \displaystyle{\sup_{\|u\|=1} \|USU^*(u)\|}\\
&=& \displaystyle{ \sup_{\|u\|=1} \|SU^*(u)\|}\\
&=& \|SU^*\| = \|US\|\\
&=&  \displaystyle{\sup_{\|u\|=1} \|US(u)\|}\\
&=&  \displaystyle{\sup_{\|u\|=1} \|S(u)\| = \|S\|}.
\end{array}$$
And: $$\begin{array}{rcl}
\|(USU^*)^{-1}\| = \|US^{-1}U^*\| &=&\displaystyle{ \sup_{\|u\|=1} \|US^{-1}U^*(u)\|}\\
&=&\displaystyle{\sup_{\|u\|=1} \|S^{-1}U^*(u)\|}\\
&=& \displaystyle{\|S^{-1}U^*\| = \|US^{-1}\|}\\
&=&\displaystyle{ \sup_{\|u\|=1} \|US^{-1}(u)\|}\\
&=&\displaystyle{ \sup_{\|u\|=1} \|S^{-1}(u)\| = \|S^{-1}\|}.
\end{array}$$
\end{proof}

In the following result, we study $\{Pu_i\}_{i\in I}$ where $\{u_i\}_{i\in I}$ is a frame for $\mathcal{H}$ and $P$ is an orthogonal projection of $\mathcal{H}$.
\begin{proposition}
Let \( (u_i)_{i\in I} \) be a frame of $\mathcal{H} $ with frame bounds $A$ and $B$, and let $P$ be an orthogonal projection of $ \mathcal{H} $. Then $\{Pu_i\}_{i \in I}$ is a frame of $P(\mathcal{H})$ with bounds $A$ and $B$.
\end{proposition}
\begin{proof}
Let $u\in P(\mathcal{H})$, then $Pu=u$ and $\langle Pu_i,u\rangle =\langle u_i,Pu\rangle=\langle u_i,u\rangle$.
\end{proof}
\begin{corollary}
An orthogonal projection transforms a Parseval frame into a Parseval frame.\\
\end{corollary}

The following theorem provides conditions under which one frame is the image of another under a right $\mathbb{H}$-linear bounded operator.
\begin{theorem}\label{thm16}
Let $(u_i)_{i\in I}$ and $(v_i)_{i \in I}$ be two frames of $\mathcal{H}$ with pre-frame operators $T_1$ and $T_2$, respectively.  
We define the relation $L$ by:
\[
L\left(\sum_{i=1}^{\infty} u_iq_i\right) = \sum_{i=1}^{\infty}  v_iq_i
\]
for any sequence $(q_i)_{i \in I}\subset \mathbb{H}$ that is zero except possibly for a finite number of elements.

Then, the following statements are equivalent:
\begin{enumerate}
\item $L$ is well defined as a right $\mathbb{H}$-linear bounded opeartor on $\mathcal{H}$.
\item $\operatorname{Ker}(T_1) \subset \operatorname{Ker}(T_2)$.
\end{enumerate}
\end{theorem}
\begin{proof} Denote by $\{e_i\}_{i\in I}$ the standard Hilbert basis of $\ell^2(\mathbb{H})$.
\begin{enumerate}
\item[] $1.\Longrightarrow 2.)$ Assume that $L$ is well a right  $\mathbb{H}$-linear bounded operator on $\text{span}\{u_i\}_{i\in I}$, then it can be uniquely extended to a right $\mathbb{H}$-linear bounded operator on $\mathcal{H}$, denoted also by $L$, since $\text{span}\{u_i\}_{i\in I}$ is dense in $\mathcal{H}$. Let $q:=\{q_i\}_{i\in I}\subset \ell^2(\mathbb{H})$, we have:
$$\begin{array}{rcl}
T_1(q)=0 &\Longrightarrow&  \displaystyle{\sum_{i\in I} u_i\langle e_i,q\rangle }=0\\
&\Longrightarrow& L\left(\displaystyle{\sum_{i\in I} u_i\langle e_i,q\rangle }\right)=0\\
&\Longrightarrow&\displaystyle{\sum_{i\in I}Lu_i\langle e_i,q\rangle }=0\\
&\Longrightarrow& \displaystyle{\sum_{i\in I} v_i\langle e_i,q\rangle }\\
&\Longrightarrow& T_2(q)=0.
\end{array}$$
Hence, $\operatorname{Ker}(T_1) \subset \operatorname{Ker}(T_2)$.\\
Conversely, assume that $\operatorname{Ker}(T_1) \subset \operatorname{Ker}(T_2)$. Let's show, first, that $L$ is well defined. Let $p:=\{p_i\}_{i\in I},\; q:=\{q_i\}_{i\in I}\subset \mathbb{H}$ be  zeros except possibly for a finite number of elements. , we have: 
$$\begin{array}{rcl}
\displaystyle{\sum_{i\in I}u_ip_i}=\displaystyle{\sum_{i\in I}u_iq_i}&\Longrightarrow& T_1(p)=T_1(q)\\
&\Longrightarrow& T_1(p-q)=0\\
&\Longrightarrow& T_2(p-q)=0\\
&\Longrightarrow& T_2(p)=T_2(q)\\
&\Longrightarrow& \displaystyle{\sum_{i\in I}v_ip_i}=\displaystyle{\sum_{i\in I}v_iq_i}
\end{array}$$
Hence, $L$ is well defined on $\text{span}\{u_i\}_{i\in I}$ to $\text{span}\{u_i\}_{i\in I}$. It is clear that $L$ is right $\mathbb{H}$-linear. Let's show, now, that $L$ is bounded. Let $T=T_1$ or $T_2$, it is clear that $T_{|kert(T)^\perp}:Ker(T)^\perp\rightarrow H$ is invertible since $T:\ell^2(\mathbb{H})\rightarrow H$ is sujective. Then, by the open map theorem \ref{thm6}, $(T_{|kert(T)^\perp})^{-1}$ is also bounded, and then for all $q\in ker(T)^\perp$, we have: 
$$\alpha \|q\|\leq \|T_{|ker(T)^\perp}q\|\leq \beta\| q\|,$$
where $\alpha=\displaystyle{\frac{1}{\|(T_{|ker(T)^\perp})^{-1}\|}}$ and $\beta=\| T_{|ker(T)^\perp}\|$. Let, now, $q=\{q_i\}_{i\in I}\subset \mathbb{H}$ be zero except possibly for a finite number of elements and denote $q_1:=Proj_{ker(T_1)^\perp}q$ and $q_2:=Proj_{ker(T_2)^\perp}q$, where $Proj_F$ is the orthogonal projection onto $F$ for $F\subset \ell^2(\mathbb{H})$ closed. We have: 
$$\begin{array}{rcl}
\left\| L\left( \displaystyle{\sum_{i\in I}u_iq_i}\right)\right\|&=&\| \displaystyle{\sum_{i\in I}v_iq_i}\|\\
&=&\| T_{2_{|ker(T_2)^\perp}}(q_2)\|\\
&\leq&\beta_2\|q_2\|\\
&\leq& \beta_2 \|q_1\| \;\; \text{ since } ker(T_1)\subset ker(T_2)\\
&\leq& \displaystyle{\frac{\beta_2}{\alpha_1}}\|T_{1_{ker(T_1)^\perp}}(q_1)\|\\
&=&\displaystyle{\frac{\beta_2}{\alpha_1}}\| T_1(q)\|\\
&=&\displaystyle{\frac{\beta_2}{\alpha_1}}\left\| \displaystyle{\sum_{i\in I}u_iq_i}\right\|
\end{array}$$
Then $L: \text{span}\{u_i\}_{i\in I}\rightarrow \text{ span}\{v_i\}_{i\in I}$ is bounded.
\end{enumerate}
\end{proof}

\begin{definition}
Let $\{u_i\}_{i\in I}$ and $\{v_i\}_{i\in I}$ be two sequences of vectors in $\mathcal{H}$.  
We say that $\{u_i\}_{i\in I}$ and $\{v_i\}_{i\in I}$ are equivalent, and we write $\{u_i\}_{i\in I} \approx \{v_i\}_{i\in I}$, if there exists a right $\mathbb{H}$-linear bounded invertible operator $L : \mathcal{H} \to \mathcal{H}$ such that for all $i\in I$, $L(u_i) = v_i$.
\end{definition}
\begin{remark}
In view of the open map theorem \ref{thm6}, we have that: $$\{u_i\}_{i\in I} \approx \{v_i\}_{i\in I} \Longleftrightarrow \{v_i\}_{i\in I} \approx \{u_i\}_{i\in I}.$$
\end{remark}
\begin{example}
Each frame on $\mathcal{H}$ is equivalent to a Parseval frame. Indeed, we have already seen, in Corollary \ref{cor1}, that if $\{u_i\}_{i\in I}$ is a frame for $\mathcal{H}$ with the frame operator $S$, then $\{S^{\frac{-1}{2}}u_i\}_{i\in I}$ is a Parseval frame for $\mathcal{H}$.\\
\end{example}
The following theorem provides sufficient and necessary conditions under which two frames are equivalent.
\begin{theorem}
Let $(u_i)_{i\in I}$ and $(v_i)_{i \in I}$ be two frames of $\mathcal{H}$ with pre-frame operators $T_1$ and $T_2$, respectively. Then the following statements are equivalent.
\begin{enumerate}
\item $\{u_i\}_{i\in I}\approx \{v_i\}_{i\in I}$.
\item $ker(T_1)=ker(T_2)$.
\end{enumerate}
\end{theorem}
\begin{proof}
Assume that $\{u_i\}_{i\in I}\approx \{v_i\}_{i\in I}$, then there exists a right $\mathbb{H}$-linear, bounded and invertible operator $L$ such that $Lu_i=v_i$ for all $i\in I$. Then, by Theorem \ref{thm16}, $ker(T_1)\subset ker(T_2)$. By the open map theorem \ref{thm6}, $L^{-1}$ is also a right $\mathbb{H}$-linear bounded operator and since $v_i=L^{-1}u_i$, then by Theorem \ref{thm16}, $ker(T_2)\subset ker(T_1)$. Hence $ker(T_1)=ker(T_2)$. Conversely, assume that $ker(T_1)=ker(T_2)$, then by Theorem \ref{thm16}, $L_1:u_i\to v_i$ and $L_2:v_i\to u_i$ are two right $\mathbb{H}$-linear bounded operators. Moreover, $L_1L_2=L_2L_1=I$, where $I$ is the identity operator of $\mathbb{B}(\mathcal{H})$. Hence $\{u_i\}_{i\in I}$ and $\{v_i\}_{i\in I}$ are equivalent.
\end{proof}

We have seen  that each frame for $\mathcal{H}$ defines a self-adjoint, positive and invertible right $\mathbb{H}$-linear bounded operator which is its frame operator. The next theorem investigate the reciproque, i.e., Given a self-adjoint, positive invertible operator $L\in \mathbb{B}(\mathcal{H})$, is there a frame for $\mathcal{H}$ for which the frame operator is $L$.
\begin{theorem}
Let $L\in \mathbb{B}(\mathcal{H})$ a self-adjoint, positive and invertible opeartor on $\mathcal{H}$. Then, there exists a frame for $\mathcal{H}$ for which the frame operator is $L$.
\end{theorem}
\begin{proof}
Let $\{v_i\}_{i\in I}$ be a Hilbert basis for $\mathcal{H}$ and define for all $i\in I$, $u_i:=L^{\frac{1}{2}}v_i$. By Theorem \ref{thm15}, $\{u_i\}_{i\in I}$ is a frame for $\mathcal{H}$ and, by Proposition \ref{prop8}, its frame opeartor is $L^{\frac{1}{2}}IL^{\frac{1}{2}}=L$, where $I$ is the identity operator of $\mathbb{B}(\mathcal{H})$.
\end{proof}

The next theorem shows that the correspondence between Bessel sequences and transform operators is bijective.
\begin{theorem}
Let \( \mathcal{B} \) denote the set of all Bessel sequences in \( \mathcal{H} \) indexed by a
countable set \( I \). Then the map:
$$\begin{array}{rcl}
\theta : \mathcal{B} &\rightarrow& \mathbb{B}(\mathcal{H}, \ell^2(\mathbb{H}))\\
\{u_i\}_{i \in I} &\mapsto& \theta_u,
\end{array}$$
where \( \theta_u \) is the transform operator of \( \{u_i\}_{i \in I} \), is bijective. Moreover,
$$\begin{array}{rcl}
\theta^{-1} : \mathbb{B}(\mathcal{H}, \ell^2(\mathbb{H})) &\rightarrow& \mathcal{B}\\
L &\mapsto& \{L^*(v_i)\}_{i \in I},
\end{array}$$
where \( \{v_i\}_{i \in I} \) is the standard Hilbert basis for \( \ell^2(\mathbb{H}) \).
\end{theorem}
\begin{proof}
Let $u=\{u_i\}_{i\in I}$, $w=\{w_i\}_{i\in I}$ be two Bessel sequences for $\mathcal{H}$. We have:
$$\begin{array}{rcl}
\theta_u=\theta_w &\Longrightarrow& \; \forall x\in \mathcal{H},\;\forall i \in I,\; \langle u_i,x\rangle=\langle w_i,x\rangle\\
&\Longrightarrow& \forall i\in I,\; u_i=w_i\\
&\Longrightarrow& u=w.
\end{array}$$
Then, $\theta$ is injective. Let $L\in \mathbb{B}(\mathcal{H}, \ell^2(\mathbb{H}))$, and set $f=\{L^*v_i\}_{i\in I}$. We have for all $u\in \mathcal{H}$, $\theta_f u=\{\langle L^*v_i,u\rangle \}_{i\in I}=\{\langle v_i,Lu\rangle \}_{i\in I}=Lu$ since $\{v_i\}_{i\in I}$ is the standard Hilbert basis for $\ell^2(\mathbb{H})$. Hence, $\theta$ is surjective, thus is bijective. Moreover: 
$$\begin{array}{rcl}
\theta^{-1} : \mathbb{B}(\mathcal{H}, \ell^2(\mathbb{H})) &\rightarrow& \mathcal{B}\\
L &\mapsto& \{L^*(v_i)\}_{i \in I},
\end{array}$$
where \( \{v_i\}_{i \in I} \) is the standard Hilbert basis for \( \ell^2(\mathbb{H}) \).
\end{proof}
\begin{remark}
The set $\mathbb{B}(\mathcal{H},\ell^2(\mathbb{H}))$ is a $\mathbb{R}$-vector space and the well known norm of right $\mathbb{H}$-linear bounded operators is a $\mathbb{R}$-norm on $\mathbb{B}(\mathcal{H},\ell^2(\mathbb{H}))$. We set for for all $u:=\{u_i\}_{i\in I}\in \mathcal{B}$, $\|u\|:=\|\theta_u\|$, then $\theta$ is a $\mathbb{R}$-linear, bounded invertible operator. Moreover, $\theta$ is isometric.
\end{remark}

\medskip

	\section*{Acknowledgments}
	It is my great pleasure to thank the referee for his careful reading of the paper and for several helpful suggestions.
	
	\section*{Ethics declarations}
	
	\subsection*{Availablity of data and materials}
	Not applicable.
	\subsection*{Conflict of interest}
	The author declares that hes has no competing interests.
	\subsection*{Fundings}
	Not applicable.


\medskip
	
	\begin{thebibliography}{99}

\bibitem{1}
S.L. Adler, \textit{Quaternionic Quantum Mechanics and Quantum Fields}, Oxford University Press, New York, 1995.

\bibitem{2}
R. Balan, P.G. Casazza, and D. Edidin, On signal reconstruction without phase, \textit{Appl. Comp. Harm. Anal.} 20 (2006), 345--356.

\bibitem{3}
J. Benedetto, A. Powell, and O. Yilmaz, Sigma-Delta quantization and finite frames, \textit{IEEE Trans. Inform. Theory} 52 (2006), 1990--2005.


\bibitem{4}
P.G. Casazza, The art of frame theory, \textit{Taiwanese J. of Math.} 4 (2000), No. 2, 129--201.

\bibitem{5}
P.G. Casazza and G. Kutyniok, Frames of subspaces, \textit{Wavelets, Frames and Operator Theory} (College Park, MD, 2003), Contemp. Math., 345, Amer. Math. Soc., Providence, RI, 2004, 87--113.

\bibitem{6}
Q. Chen, P. Dang, and T. Qian, A frame theory of Hardy spaces with the quaternionic and the Clifford algebra setting, \textit{Adv. Appl. Clifford Algebras} 27 (2017), 1073--1101.


\bibitem{7}
O. Christensen, \textit{An Introduction to Frames and Riesz Bases}, Applied and Numerical Harmonic Analysis, Birkhäuser Boston, Inc., Boston, MA, 2003.

\bibitem{8}
I. Daubechies, A. Grossmann, and Y. Meyer, Painless non-orthogonal expansions, \textit{J. Math. Physics} 27 (1986), 1271--1283.

\bibitem{9}
R.J. Dun and A.C. Schaeffer, A class of non-harmonic Fourier series, \textit{Trans. Amer. Math. Soc.} 72 (1952), 341--366.

\bibitem{10}
R. Ghiloni, V. Moretti, and A. Perotti, Continuous slice functional calculus in quaternionic Hilbert spaces, \textit{Rev. Math. Phys.} 25 (2013), 1350006.

\bibitem{11}
S.K. Sharma and S. Goel, Frames in quaternionic Hilbert spaces, \textit{J. Math. Phys. Anal. Geom.} 15 (2019), 395--411.

\bibitem{12}
R.W. Heath and A.J. Paulraj, Linear dispersion codes for MIMO systems based on frame theory, \textit{IEEE Trans. Signal Process.} 50 (2002), 2429--2441.

\bibitem{13}
M. Khokulan, K. Thirulogasanthar, and S. Srisatkunarajah, Discrete frames on finite dimensional quaternion Hilbert spaces, Proceedings of Jaffna University International Research Conference (JUICE 2014).



\end{thebibliography}
\end{document}